\documentclass[letterpaper, 10 pt, conference]{ieeeconf}
\IEEEoverridecommandlockouts
\usepackage{graphics}
\usepackage{epsfig}
\usepackage{amsmath}
\usepackage{amssymb}
\usepackage{xcolor}
\let\labelindent\relax
\usepackage{enumitem}
\usepackage{multirow}
\usepackage{url}
\usepackage{lipsum}
\usepackage[noadjust]{cite} 
  
\usepackage{amsthm}
\theoremstyle{definition}
\newtheorem{lemma}{\normalfont \bfseries Lemma}

\newtheorem{theorem}{\normalfont\bfseries Theorem}
\newtheorem{assumption}{\normalfont\bfseries Assumption}

\newtheorem{definition}{\normalfont\bfseries Definition}

\newtheorem{corollary}{\normalfont\bfseries Corollary}

\newtheorem{remark}{\normalfont\bfseries Remark}

\newcommand{\derp}[2]{\frac{\partial #1}{\partial #2}}
\newcommand{\R}{\mathbb{R}}
\newcommand{\X}{\mathbb{R}^n}
\newcommand{\U}{\mathbb{R}^m}
\newcommand{\Y}{\mathbb{R}^p}
\newcommand{\C}{\mathcal{C}}
\renewcommand{\S}{\mathcal{S}}
\newcommand{\E}{\mathcal{E}}
\newcommand{\K}{\mathcal{K}}
\newcommand{\Ke}{\K^{\rm e}}

\newcommand{\bzero}{\mathbf{0}}
\newcommand{\bx}{\mathbf{x}}
\renewcommand{\bf}{\mathbf{f}}
\newcommand{\bg}{\mathbf{g}}
\newcommand{\bu}{\mathbf{u}}
\newcommand{\bk}{\mathbf{k}}
\newcommand{\by}{\mathbf{y}}
\newcommand{\bkappa}{\boldsymbol{\kappa}}
\newcommand{\bGamma}{\boldsymbol{\Gamma}}

\newcommand{\bb}{\mathbf{b}}

\title{\LARGE \textbf{
Activated Backstepping with Control Barrier Functions \\ for the Safe Navigation of Automated Vehicles
}}

\author{Laszlo Gacsi, Max H. Cohen, and Tamas G. Molnar%
\thanks{*The material contained in this document is based upon work supported by a National Aeronautics and Space Administration (NASA) grant or cooperative agreement. Any opinions, findings, and conclusions or recommendations expressed in this material are those of the author(s) and do not necessarily reflect the views of NASA. This work was supported through a NASA grant awarded to the Kansas NASA EPSCoR Program.}%
\thanks{L. Gacsi and T. G. Molnar are with the Department of Mechanical Engineering, Wichita State University, Wichita, KS 67260, USA,
{\tt\small lxgacsi@shockers.wichita.edu, tamas.molnar@wichita.edu}.}%
\thanks{M. H. Cohen is with the Department of Mechanical and Civil Engineering, California Institute of Technology, Pasadena, CA 91125, USA,
{\tt\small maxcohen@caltech.edu}.}%
\vspace{-1mm}
}

\begin{document}

\maketitle
\thispagestyle{empty}
\pagestyle{empty}

\begin{abstract}
This paper introduces a novel safety-critical control method through the synthesis of control barrier functions (CBFs) for systems with high-relative-degree safety constraints.
By extending the procedure of CBF backstepping, we propose {\em activated backstepping}---a constructive method to synthesize valid CBFs.
The novelty of our method is the incorporation of an activation function into the CBF, which offers less conservative safe sets in the state space than standard CBF backstepping.
We demonstrate the proposed method on an inverted pendulum example, where we explain the underlying geometric meaning in the state space 
and provide comparisons with existing CBF synthesis techniques.
Finally, we implement our method to achieve collision-free navigation for automated vehicles using a kinematic bicycle model in simulation.
\end{abstract}

\section{INTRODUCTION}
\label{sec:intro}

Safety is one of the most important features in many modern control systems, including automated vehicles (AVs).
A promising way to achieve provably safe behavior for these systems is through the use of control barrier functions (CBFs) \cite{AmesXuGriTab2017}. CBFs have a wide range of applications related to the collision-free navigation of AVs.
For example,
longitudinal controllers (like adaptive cruise control) may leverage CBFs to keep the distance between AVs and other vehicles above a safe value \cite{ames2014control}, even in dynamically changing environments \cite{molnar2022safety}, which has also been demonstrated experimentally~\cite{Alan2023AV}.
CBFs 
also
enable safe lateral control for scenarios like obstacle and collision avoidance
\cite{black2023futurefocused, goswami2024collision, haraldsen2024safety, kim2024control}, lane-keeping \cite{jiang2024safety}, lane change~\cite{hu2023safety}, and roundabout crossing~\cite{abduljabbar2021cbfbased}.


To achieve collision-free navigation, AVs must satisfy 
position constraints
using control inputs that 
enter at the acceleration level.
There exist methods to address such high-relative-degree constraints using CBFs.
The first method is \textit{high-order CBFs} \cite{Nguyen2016, xiao2022hocbf}, which became popular due to their 
simplicity,
although they require additional assumptions to obtain a valid CBF.
Another constructive method is \textit{backstepping} \cite{taylor2022safebackstepping, cohen2024constructive, cohen2024reduced}, which 
produces
valid CBFs 
using
an auxiliary, virtual safe controller, but increases the complexity of the control design.
More recently, \cite{ong2024rectified} introduced \emph{rectified CBFs} (ReCBFs) to address high-relative-degree constraints with less complexity. 
ReCBFs incorporate
an activation function into the CBF along with a carefully selected parameter to ensure the Lipschitz continuity of the resulting controller.

The main contribution of this paper is a novel CBF construction that unites the strengths of backstepping \cite{taylor2022safebackstepping,cohen2024constructive} and ReCBFs \cite{ong2024rectified}.
The end result is {\em activated backstepping CBFs (ABCs)}---valid CBFs that offer a larger but still valid safe set compared to that obtained via traditional backstepping, while ensuring the Lipschitz continuity of these controllers
under weaker conditions than those produced by ReCBFs.
We illustrate the benefits of ABCs and compare them with the aforementioned methods using an inverted pendulum example.
Finally, we apply ABCs to the safe navigation of AVs, considering a kinematic bicycle model that endeavors to bypass an obstacle, as shown in Fig.~\ref{fig:firstpage}.

\begin{figure}[t!]
\centering
\includegraphics[width=.47\textwidth]{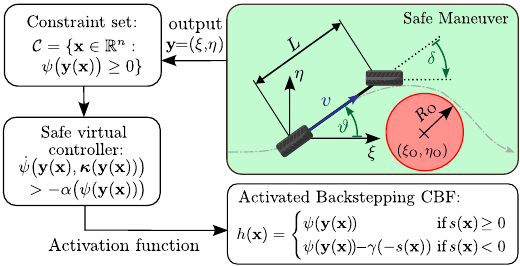}
\vspace{0mm}
\caption{Summary of the proposed safety-critical control framework.}
\vspace{-4mm}
\label{fig:firstpage}
\end{figure}


\section{BACKGROUND}
\label{sec:CBF}

\subsection{Control Barrier Functions}

We consider nonlinear control systems in affine form:
\begin{equation}
    \dot{\bx} = \bf(\bx) + \bg(\bx) \bu,
	\label{eq:system}
\end{equation}
where ${\bx \in \X}$ is the state, ${\bu \in \U}$ is the input, and ${\bf: \X \to \R^n}$, ${\bg: \X \to \R^{n\times m}}$ are smooth\footnote{We say that a function is smooth if it is continuously differentiable as many times as necessary.
For smooth functions ${\alpha : \R \to \R}$, ${h : \X \to \R}$, and ${\by: \X \to \Y}$, we denote the derivative as ${\alpha' : \R \to \R}$, gradient as ${\derp{h}{\bx} : \X \to \X}$, and Jacobian as ${\derp{\by}{\bx} : \X \to \R^{p \times n}}$, respectively.
With a smooth vector field ${\bf: \X \to \X}$, we define the Lie derivatives ${L_{\bf}h(\bx) = \derp{h}{\bx}(\bx) \cdot \bf(\bx)}$ and ${L_{\bf}\by(\bx) = \derp{\by}{\bx}(\bx) \cdot \bf(\bx)}$.} vector and matrix functions, respectively. A locally Lipschitz controller ${\bk: \X \to \U}$, ${\bu = \bk(\bx)}$, leads to the closed-loop system:
\begin{equation}
    \dot{\bx} = \bf(\bx) + \bg(\bx) \bk(\bx).
	\label{eq:closedloop}
\end{equation}

We describe the safety of this system using a {\em safe set} ${\S \subset \X}$ in the state space.
We say that system~\eqref{eq:closedloop} is safe w.r.t.~set $\S$ if for each initial condition ${\bx(0) \in \S}$ the unique solution of~\eqref{eq:closedloop} satisfies ${\bx(t)\in\S}$ for all time.
We define safe sets by a continuously differentiable function ${h : \X \to \R}$:
\begin{equation}
    \S = \{ \bx \in \X : h(\bx) \geq 0 \}.
    \label{eq:safeset}
\end{equation}

\begin{definition}[\cite{AmesXuGriTab2017}]
The function $h$ from \eqref{eq:safeset} is a {\em control barrier function} for~\eqref{eq:system} on $\S$ if there exists an extended class-$\K$ function\footnote{Function ${\alpha : (-b,a) \to \R}$, ${a,b>0}$ is of extended class-$\K$ (${\alpha \in \Ke}$) if it is continuous, strictly increasing, and ${\alpha(0)=0}$.}
${\alpha \in \Ke}$ and an open set $\E\supset\S$ such that:
\begin{equation}
    \sup_{\mathbf{u} \in \U} \dot{h}(\mathbf{x,u}) > - \alpha \big( h(\bx) \big),
\label{eq:CBF_condition}
\end{equation}
for all ${\bx \in \E}$, where ${\dot{h}(\bx,\bu) = L_{\bf}h(\bx) + L_{\bg}h(\bx) \bu}$.
\end{definition}

\begin{theorem}[\cite{AmesXuGriTab2017}] \label{theo:CBF}
\textit{
If $h$ is a CBF for~\eqref{eq:system} on $\S$, then any locally Lipschitz controller ${\bk : \E\ \to \U}$ that satisfies: 
\begin{equation}
    \dot{h} \big( \bx, \mathbf{k(x)} \big) \geq - \alpha \big( h(\bx) \big)
\label{eq:safety_condition}
\end{equation}
for all ${\bx \in \S}$ renders~\eqref{eq:closedloop} safe w.r.t.~$\S$.
}
\end{theorem} 

Condition~\eqref{eq:safety_condition} facilitates the design of controllers that guarantee safety.
For example, solving the optimization problem:
\begin{equation}
\begin{aligned}
    \bk(\bx) = \underset{\bu \in \U}{\operatorname{argmin}} & \quad \| \bu - \bk_{\rm d}(\bx) \|^2_{\bGamma} \\[-6pt]
    \text{s.t.} & \quad \dot{h}(\bx,\bu) \geq -\alpha \big( h(\bx) \big),
    \label{eq:QP}
\end{aligned}
\end{equation}
produces a safe controller that minimally modifies a
smooth desired controller ${\bk_{\rm d} : \X \to \U}$.
Here, ${\bGamma={\rm diag} \{\Gamma_1, \ldots, \Gamma_m\}}$, ${\Gamma_i>0}$ is a matrix that weighs the input components and ${\|\bu\|^2_{\bGamma} = \bu^\top \bGamma \bu}$.
The optimization problem~\eqref{eq:QP}, also called the {\em safety filter}, is feasible and the controller is locally Lipschitz continuous if $h$ is a CBF with a \emph{strict} inequality~\cite{jankovic2018robust}.
The solution of~\eqref{eq:QP} is~\cite{cohen2023smooth, molnar2025fixedwing}:
\begin{align} \label{eq:safetyfilter}
    \bk(\bx) & = \bk_{\rm d}(\bx) + \lambda \big( a(\bx), \|\bb(\bx)\|^2_{\bGamma} \big) \bb(\bx), \\
    a(\bx) & = \dot{h} \big( \bx, \bk_{\rm d}(\bx) \big) + \alpha \big( h(\bx) \big), \;\;
    \bb(\bx) = \bGamma^{-1} L_{\bg}h(\bx)^\top, \nonumber
\end{align}
with:
\begin{equation}
    \lambda(a,b) =
    \begin{cases}
        0 & \text{if } b \leq 0, \\
        \max\{0,-a/b\} & \text{if } b > 0.
    \end{cases}
    \label{eq:lambda_QP}
\end{equation}
A smooth over-approximation of $\lambda$ in~\eqref{eq:lambda_QP} yields a smooth safe controller in~\eqref{eq:safetyfilter}.
For example, the half-Sontag formula:
\begin{equation}
    \lambda(a,b) =
    \begin{cases}
        0 & \text{if } b = 0, \\
        \frac{-a + \sqrt{a^2 + \sigma b^2}}{2 b} & \text{if } b \neq 0,
    \end{cases}
    \label{eq:lambda_smooth}
\end{equation}
with a smoothing parameter ${\sigma > 0}$, leads to a smooth controller in~\eqref{eq:safetyfilter} that strictly satisfies ${\dot{h} \big( \bx, \bk(\bx) \big) > -\alpha \big( h(\bx) \big)}$; see~\cite{cohen2024reduced, cohen2023smooth,cohen2024constructive}.
Note that $\lambda$ in~\eqref{eq:lambda_QP} is recovered for ${\sigma \to 0}$.

The success of safety-critical control is conditioned on whether $h$ is a valid CBF satisfying~\eqref{eq:CBF_condition}.
The following lemma helps to verify the validity of a candidate CBF.

\begin{lemma}[\cite{jankovic2018robust}]\label{lemma:cbf}
    A function $h$ as in~\eqref{eq:safeset} is a CBF for~\eqref{eq:system} on $\S$ if and only if
    there exists ${\alpha \in \Ke}$ and $\E\supset\S$ such that:
    \begin{equation}\label{eq:Lgh=0}
        L_{\bg}h(\bx) = \bzero \implies L_{\bf}h(\bx) > - \alpha \big( h(\bx) \big),
        \quad \forall \bx \in \E.
    \end{equation}
\end{lemma}

Lemma~\ref{lemma:cbf} implies that $h$ is a CBF if ${L_{\bg}h(\bx) \neq \bzero}$ for all ${\bx \in \E}$.
In such a case, the control input $\bu$ appears in the first time derivative of $h$ for all $\bx\in\E$, and we say that $h$ has {\em relative degree one}.
In contrast, if ${L_{\bg}h(\bx) = \bzero}$ for all ${\bx \in \E}$, then $h$ is not a CBF (unless the system is safe without control, i.e., ${L_{\bf}h(\bx) > - \alpha \big( h(\bx) \big)}$ for all ${\bx \in \E}$).
In this case, we say that $h$ has {\em higher relative degree} because the control input $\bu$ shows up in higher derivatives of $h$.

\subsection{Safety Constraints with High Relative Degree}

We revisit techniques to generate CBFs for higher relative degree cases.
We focus on {\em relative degree two}, although these methods can be recursively applied to relative degrees higher than two.
Relative degree two often occurs in mechanical systems where the safety constraints are on position (or configuration) and the inputs are forces (or torques) affecting the second derivative of position due to Newton's second law.

Consider a smooth output ${\by : \X \to \Y}$ that selects the states (e.g., position) relevant to safety, and a {\em constraint set} $\C$ defined by a smooth {\em constraint function} ${\psi : \Y \to \R}$:
\begin{equation}
    \C = \{\bx \in \X : \psi \big( \by(\bx) \big) \geq 0\}.
    \label{eq:constraintset}
\end{equation}

\begin{assumption} \label{assum:reldeg}
The output $\by$ satisfies:
\begin{equation}
    L_{\bg} \by(\bx) = \bzero, \quad
    {\rm rank} \big( L_{\bg} L_{\bf} \by(\bx)\big) = p,
    \quad \forall \bx \in \E.
\end{equation}
\end{assumption}

Assumption \ref{assum:reldeg} implies that the output $\by$ has relative degree two: the control input $\bu$ shows up in its second time derivative.
We use the following notation for first time derivatives:
\begin{equation}
    \dot{\by}(\bx) = L_{\bf} \by(\bx), \quad\!\!
    \dot{\psi} \big( \by(\bx), \dot{\by}(\bx) \big) = \derp{\psi}{\by}(\by(\bx)) \cdot \dot{\by}(\bx).
\end{equation}
Since $\dot{\psi}$ is independent of the input $\bu$, the constraint function $\psi$ is not a CBF, in general.
The following approaches generate a candidate CBF $h$ from the constraint function $\psi$.

\subsubsection{High-order CBF}

The most popular tool to address high-relative-degree safety constraints is {\em high-order control barrier functions (HOCBFs)}~\cite{Nguyen2016, xiao2022hocbf}, with $h$ defined by:
\begin{equation}
    h(\bx) = \dot{\psi} \big( \by(\bx), \dot{\by}(\bx) \big) + \alpha \big( \psi(\by(\bx)) \big).
    \label{eq:hocbf}
\end{equation}
This provides a simple and efficient way of safety-critical control, as enforcing~\eqref{eq:safety_condition} leads to safety w.r.t.~${\mathcal{S} \cap \mathcal{C}}$; see~\cite{xiao2022hocbf}.
The limitation of HOCBFs is that $h$ in~\eqref{eq:hocbf} is often not a valid CBF satisfying~\eqref{eq:Lgh=0}
when $h$ has a weak relative degree, i.e., if ${L_{\mathbf{g}}h(\bx) =\bzero}$ for some, but not all, ${\bx \in \E}$; see \cite{ong2024rectified}.

\subsubsection{Rectified CBF}
To overcome the limitations of HOCBFs,
~\cite{ong2024rectified} proposed the so-called {\em ReCBF} defined by:
\begin{equation}
\begin{aligned}
    h(\bx) & = \psi(\by(\bx)) - {\rm ReLU} \!\Big(\! -\!\gamma \big( r(\bx) - \varepsilon \big) \!\Big), \\
    r(\bx) & = \dot{\psi} \big( \by(\bx), \dot{\by}(\bx) \big) + \alpha \big( \psi(\by(\bx)) \big),
    \label{eq:ReCBF}
\end{aligned}
\end{equation}
where ${{\rm ReLU}(\cdot)=\max\{\cdot,0\}}$ is the rectified linear unit,
${\gamma \in \Ke}$ is smooth with ${\gamma'(s)=0 \iff s=0}$,
and ${\varepsilon>0}$ is a parameter. Note that $r(\bx)$ in~\eqref{eq:ReCBF} is identical to $h(\bx)$ in~\eqref{eq:hocbf}.
It was shown that $h$ is a CBF satisfying~\eqref{eq:Lgh=0} if:
\begin{equation}
    L_{\bg}L_{\bf}\psi(\bx) \!=\! \bzero \!\implies\! \dot{\psi} \big( \by(\bx), \dot{\by}(\bx) \big) \!+\! \alpha \big( \psi(\by(\bx)) \big) \!\geq\! \varepsilon,
    \label{eq:ReCBF_condition}
\end{equation}
and enforcing~\eqref{eq:safety_condition} guarantees safety w.r.t.~${\S \subset \C}$; see~\cite{ong2024rectified}.
Note that $\varepsilon$ must be carefully selected: if $\varepsilon$ is too large,~\eqref{eq:ReCBF_condition} is violated, and if $\varepsilon$ is too small, the input may chatter (as shown below) since the resulting controller in~\eqref{eq:safetyfilter}-\eqref{eq:lambda_QP} may not be locally Lipschitz continuous for all ${\bx \in \E}$ when ${\varepsilon \to 0}$.

\subsubsection{Backstepping CBF}

An alternative approach for systematically constructing a valid CBF is {\em backstepping}~\cite{taylor2022safebackstepping, cohen2024constructive, cohen2024reduced}.
Backstepping relies on the following assumption.

\begin{assumption} \label{assum:singleint}
    The constraint function $\psi$ satisfies:
    \begin{equation}
        \derp{\psi}{\by}(\by(\bx)) = \bzero \implies \psi(\by(\bx))>0,
        \quad \forall \bx \in \E.
    \end{equation}
\end{assumption}

Assumption~\ref{assum:singleint} implies that $\psi$ is a CBF for the single integrator ${\dot{\by} = \bu_{\by}}$; cf.~\cite[Lem. 1]{cohen2024constructive}.
Backstepping dynamically extends a CBF for a simple system (i.e., $\psi$ for a single integrator) to a CBF for a more complex system~\eqref{eq:system} 
via~\cite{cohen2024constructive}:
\begin{equation}
    h(\bx) = \psi(\by(\bx)) - \frac{1}{2\mu} \big\| \dot{\by}(\bx) - \bkappa(\by(\bx)) \big\|^2, 
    \label{eq:backstepping}
\end{equation}
where ${\mu>0}$ is a parameter, and $\bkappa : \Y \to \R^p$ is a smooth safe virtual controller for the single integrator satisfying:
\begin{equation}
    \dot{\psi} \big( \by(\bx), \bkappa (\by(\bx)) \big) > -\alpha \big( \psi(\by(\bx)) \big), \quad
    \forall \bx \in  \E.
    \label{eq:topdyn}
\end{equation}
Such a controller can be obtained, for example, by applying the smooth safety filter in~\eqref{eq:safetyfilter} and~\eqref{eq:lambda_smooth} for the single integrator.
The {\em backstepping CBF} in~\eqref{eq:backstepping} is a valid CBF that satisfies~\eqref{eq:Lgh=0}, and enforcing~\eqref{eq:safety_condition} ensures safety w.r.t.~${\S \subset \C}$; see~\cite{cohen2024constructive}.
The advantage of backstepping is the validity of the CBF and the continuity properties of the resulting controllers, while its drawback is the additional complexity of computing $\bkappa$ and its derivative to obtain $h$ and $\dot{h}$.


\section{ACTIVATED BACKSTEPPING}
\label{sec:activated}

Although backstepping provides a constructive way to generate valid CBFs, it is often conservative in that the resulting safe set $\S$ may be much smaller than the constraint set $\C$.
To remedy this, we propose {\em activated backstepping}, an approach that yields a larger safe set, and builds on
the framework of ReCBFs~\cite{ong2024rectified}.
Specifically, 
\cite[Ex. 4]{ong2024rectified}
introduced
the main idea behind
activated backstepping 
using a simple example, but did not provide a full characterization of such an approach.
Here, we fill in these gaps by formally establishing the underlying properties of this approach and providing more detailed examples.
The advantage of ABCs over ReCBFs is that they obviate the need to choose $\varepsilon$, producing a valid CBF with chatter-free inputs by construction. 

To enlarge the safe set $\S$, we make an observation.
If
${\dot{\psi} \big( \by(\bx), \dot{\by}(\bx) \big) \geq \dot{\psi} \big( \by(\bx), \bkappa (\by(\bx)) \big)}$, then the output $\by$ evolves safer than the single integrator with controller $\bkappa$ in~\eqref{eq:topdyn}, implying safety w.r.t.~$\C$.
Hence, at these points $\psi(\by(\bx))$ is a perfectly valid CBF.
Thus, we propose the {\em activated backstepping CBF (ABC)}:
\begin{equation}
    h(\bx) =
    \begin{cases}
        \psi(\by(\bx))  & \text{if } s(\bx) \geq 0, \\
        \psi(\by(\bx)) - \gamma(-s(\bx)) & \text{if } s(\bx) < 0,
    \end{cases}
\label{eq:newCBF}
\end{equation}
where ${\gamma \in \Ke}$ is smooth with ${\gamma'(s)=0 \iff s=0}$,
and:
\begin{equation}
\begin{aligned}
    s(\bx) & = \dot{\psi} \big( \by(\bx), \dot{\by}(\bx) \big) - \dot{\psi} \big( \by(\bx), \bkappa (\by(\bx)) \big) \\
    & = \derp{\psi}{\by}(\by(\bx)) \cdot \big( \dot{\by}(\bx) - \bkappa(\by(\bx)) \big).
    \label{eq:switching}
\end{aligned}
\end{equation}
Note that $s(\bx)$ in~\eqref{eq:switching} differs from $r(\bx)$ in~\eqref{eq:ReCBF}.
The term {\em activated} indicates that the proposed CBF~\eqref{eq:newCBF}-\eqref{eq:switching} can be expressed using an activation function, the ReLU, and a rectified extended class-$\K$ function ${\Theta(\cdot)={\rm ReLU}(\gamma(\cdot))}$:
\begin{equation}
    h(\bx) \!=\! \psi(\by(\bx)) - {\rm ReLU}\big( \gamma(-s(\bx)) \big)
    \!=\! \psi(\by(\bx)) - \Theta(-s(\bx)).
    \label{eq:activated}
\end{equation}
For $\gamma$ and $\Theta$, we may choose, e.g., the quadratic function ${\gamma(s) = s^2}$ for ${s > 0}$ and the rectified quadratic unit (ReQU)  ${\Theta(s)={\rm ReQU}(s)}$ (i.e., ${\Theta(s)=s^2}$ if ${s > 0}$ and ${\Theta(s)=0}$ if ${s \leq 0}$), or we may scale to ${\Theta(s)={\rm ReQU}(s)/(2\mu)}$; cf.~\eqref{eq:backstepping}.

Next, we state that the proposed $h$ in \eqref{eq:activated} is a valid CBF.  

\begin{theorem} \label{theo:newCBF}
\textit{
Consider system \eqref{eq:system} with a smooth output ${\by : \X \to \Y}$ satisfying Assumption~\ref{assum:reldeg}, a constraint set ${\C \!\subset\! \X}$ as in~\eqref{eq:constraintset} defined by a smooth function ${\psi : \Y \to \R}$ satisfying Assumption~\ref{assum:singleint}, and a safe set ${\S \subset \X}$ as in~\eqref{eq:safeset} defined by ${h : \X \to \R}$ in \eqref{eq:activated}, where ${\bkappa : \Y \to \R^p}$ is smooth and satisfies \eqref{eq:topdyn}.
Then, $h$ is a CBF for \eqref{eq:system} on ${\S \subset \C}$.
}
\end{theorem}
\begin{proof}
First, we establish that ${\Theta(\cdot)={\rm ReLU}(\gamma(\cdot))}$ in~\eqref{eq:activated}
is continuously differentiable. Since ${\gamma \in \Ke}$ is smooth with ${\gamma'(0)=0}$, we have ${\Theta'(s) = \gamma'(s)>0}$ for ${s>0}$, ${\Theta'(s)=0}$ for ${s<0}$, and ${\Theta'(0)=0}$ for both the left and right derivatives.
Furthermore, we have ${\Theta(-s) \geq 0}$ for all ${s \in \R}$ and:
\begin{equation}
    \Theta(-s)=0 \iff
    s \geq 0 \iff
    \Theta'(-s)=0,
    \label{eq:theta_prop}
\end{equation}
due to the ReLU, ${\gamma \in \Ke}$, and ${\gamma'(s)=0 \!\iff\! s=0}$.
Because $\psi$, $\by$, and $s$ are smooth and $\Theta$ is continuously differentiable, $h$ in~\eqref{eq:newCBF} is also continuously differentiable.
Since ${\Theta(-s) \geq 0}$, we have that ${h(\bx) \leq \psi(\by(\bx))}$ for all ${\bx \in \E}$, thus ${h(\bx) \geq 0 \implies \psi(\by(\bx)) \geq 0}$ and ${\S \subset \C}$.

\begin{figure*}
\centering
\includegraphics[width=1\textwidth]{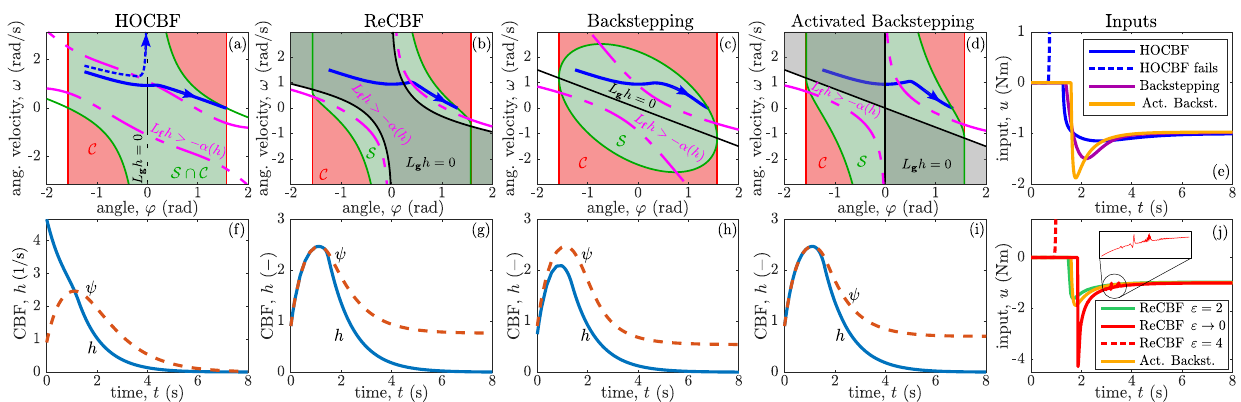}
\vspace{-7mm}
\caption{
Simulation of the inverted pendulum~\eqref{eq:pendulum} using controller~\eqref{eq:safetyfilter}-\eqref{eq:lambda_QP} with:
(a,f) high-order CBF~\eqref{eq:pend_hocbf},
(b,g) rectified CBF~\eqref{eq:pend_recbf},
(c,h) backstepping CBF~\eqref{eq:pend_backstepping}, and
(d,i) the proposed activated backstepping CBF~\eqref{eq:pend_activated}.
While the trajectories are safe in each case, the HOCBF results in infeasible control inputs for some initial conditions (e).
Compared to backstepping, activated backstepping gives a larger (unbounded) safe set at the price of slightly larger and sharper input.
Compared to the ReCBF, activated backstepping provides less sharp, chatter-free input without need for careful parameter tuning (f).
}
\vspace{-5mm}
\label{fig:pendulum}
\end{figure*}

We now prove that $h$ is a CBF using Lemma~\ref{lemma:cbf} by showing that \eqref{eq:Lgh=0} holds.
We identify $L_{\bf}h$ and $L_{\bg}h$ in~\eqref{eq:Lgh=0} by calculating the derivative of $h$ in~\eqref{eq:activated} along~\eqref{eq:system}:
\begin{equation}
\begin{aligned}
    \dot{h}(\bx,\bu)
    = \underbrace{\dot{\psi}(\by,\dot{\by}) + \Theta'(-s) L_{\bf}s}_{L_{\bf}h} + \underbrace{\Theta'(-s) L_{\bg}s}_{L_{\bg}h} \bu,
    \label{eq:hdot_activated}
\end{aligned}
\end{equation}
where functional dependencies on $\bx$ have been suppressed for ease of presentation.
By differentiating~\eqref{eq:switching} we have:
\begin{equation}
\begin{aligned}
    L_{\bf}s & = \derp{}{\by} \bigg(
    \derp{\psi}{\by}(\by) \cdot \big( \dot{\by} - \bkappa(\by) \big)
    \!\bigg) \!\cdot\! \dot{\by}
    + \derp{\psi}{\by}(\by) \!\cdot\! L_{\bf}^2\by, \\
    L_{\bg} s & = \derp{\psi}{\by}(\by) \cdot L_{\bg}L_{\bf}\by,
    \label{eq:Lfs_Lgs}
\end{aligned}
\end{equation}
with ${L_{\bf}^2\by = L_{\bf}\dot{\by}}$, ${L_{\bg}L_{\bf}\by = L_{\bg}\dot{\by}}$.
Per Assumption~\ref{assum:reldeg},
the left nullspace of $L_{\bg}L_{\bf}\by$ is the zero vector,
and~\eqref{eq:hdot_activated}-\eqref{eq:Lfs_Lgs} imply:
\begin{equation}
    L_{\bg}h = \bzero \iff \Theta'(-s) = 0\ \text{ or }\ \derp{\psi}{\by}(\by) = \bzero .
\end{equation}
If ${\derp{\psi}{\by}(\by) = \bzero}$,~\eqref{eq:switching} gives ${s=0}$, which again leads to ${\Theta'(-s) = 0}$.
Therefore, we obtain:
\begin{equation}
    L_{\bg}h = \bzero \iff \Theta'(-s) = 0 \iff s \geq 0,
    \label{eq:activated_Lgh_vs_s}
\end{equation}
where we also used~\eqref{eq:theta_prop}.
If ${\Theta'(-s) = 0}$ and ${s \geq 0}$:
\begin{equation}
\begin{aligned}
    \!\!L_{\bf}h = \dot{\psi}(\by,\dot{\by})
    \geq \dot{\psi}(\by,\bkappa(\by))
    > -\alpha(\psi(\by))
    = -\alpha(h),
\end{aligned}
\end{equation}
where we substituted ${\Theta'(-s) = 0}$ into $L_{\bf}h$ in~\eqref{eq:hdot_activated};
we used ${s \geq 0}$ and~\eqref{eq:switching};
we applied~\eqref{eq:topdyn};
and we used ${s \geq 0}$ and~\eqref{eq:newCBF}.
Thus,~\eqref{eq:Lgh=0} holds and Lemma~\ref{lemma:cbf} completes the proof.
\end{proof}

\begin{corollary}
\textit{
Consider $h$ in~\eqref{eq:activated}.
Any locally Lipschitz controller ${\bk : \E \to \U}$ that satisfies~\eqref{eq:safety_condition} for all ${\bx \in \S}$ renders~\eqref{eq:closedloop} safe w.r.t.~${\S \subset \C}$.
}
\end{corollary}

\begin{proof}
This is a direct consequence of Theorems~\ref{theo:CBF}-\ref{theo:newCBF}.
\end{proof}

\begin{remark} \label{rem:comparison}
The advantage of activated backstepping over the standard backstepping is twofold.
First, because $h$ differs from $\psi$ in~\eqref{eq:newCBF} only at safety-critical states (when ${s(\bx)<0}$), the size of the resulting safe set $\S$ becomes larger
(see the example below).
Second, when $h$ matches $\psi$ (when ${s(\bx) \geq 0}$), significant computations can be saved.
Specially, one does not need to calculate $\dot{h}$ and the derivative of $\bkappa$, as the safety filter in~\eqref{eq:safetyfilter}-\eqref{eq:lambda_QP} gives ${\bk(\bx)=\bk_{\rm d}(\bx)}$.
The ABC in~\eqref{eq:activated} also offers an advantage over the ReCBF in~\eqref{eq:ReCBF}: it strictly satisfies~\eqref{eq:CBF_condition}, i.e., ${\sup_{\mathbf{u} \in \U} \dot{h}(\bx,\bu) > -\alpha(h(\bx))}$, rather than ${\sup_{\mathbf{u} \in \U} \dot{h}(\bx,\bu) \geq -\alpha(h(\bx))}$ that holds for ReCBFs when ${\varepsilon = 0}$~\cite{ong2024rectified}.
This subtlety is important for ensuring the Lipschitz continuity of safe controllers like~\eqref{eq:QP}~\cite{jankovic2018robust}.
For ${\varepsilon=0}$, ReCBFs could yield discontinuous inputs with chatter (frequent jumps), and for small ${\varepsilon>0}$, the inputs may change rapidly (see the example below).
Thus, carefully selecting $\varepsilon$ is important.
ABCs eliminate the challenge of tuning $\varepsilon$, at the price of additional complexity of designing controller $\bkappa$.
\end{remark}

\begin{figure*}
\centering
\includegraphics[width=.95\textwidth]{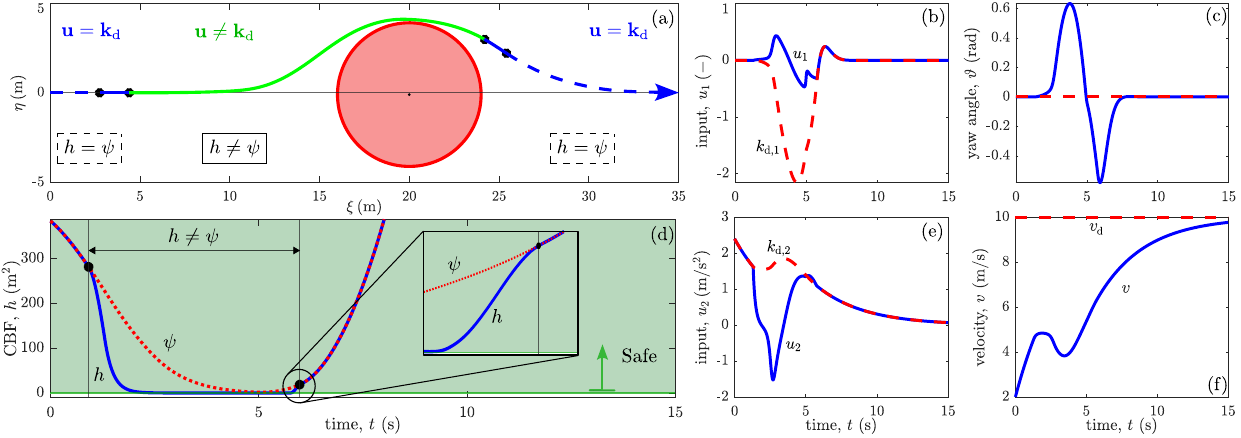}
\vspace{-3mm}
\caption{
Simulation of the kinematic bicycle model~\eqref{eq:bicycle} using controller~\eqref{eq:safetyfilter}-\eqref{eq:lambda_QP} with the proposed activated backstepping CBF~\eqref{eq:activated}.
The vehicle successfully bypasses an obstacle with desired, safe behavior.
Codes are available at https://github.com/LaszloGacsi/activated-backstepping-CBF.}

\vspace{-5mm}
\label{fig:bicycle}
\end{figure*}

\subsection{Illustrative Example}
We first demonstrate the utility of activated backstepping using an inverted pendulum example, as in~\cite{Alan2023AV}, and compare it with the other methods.
The equations of motion are:  
\begin{equation}
    \underbrace{\begin{bmatrix}
        \dot{\varphi} \\
        \dot{\omega}
    \end{bmatrix}}_{\dot{\bx}} =
    \underbrace{\begin{bmatrix}
        \omega \\
        \sin(\varphi)
    \end{bmatrix}}_{\bf(\bx)} +
    \underbrace{\begin{bmatrix}
        0 \\ 1
    \end{bmatrix}}_{\bg(\bx)} u,
    \label{eq:pendulum}
\end{equation}
where the state $\bx=(\varphi, \omega)$ includes the angle $\varphi$ measured from the upright position and the angular velocity $\omega$, while the input $u$ is a torque.
For safety, we keep the pendulum above the horizontal position so that ${-\pi/2 \leq \varphi \leq \pi/2}$.
This is captured by the output ${y(\bx) = \varphi}$ and constraint function:
\begin{equation}
    \psi(\varphi) = \frac{\pi^2}{4} - \varphi^2.
\end{equation}
Since
${\dot{y}(\bx) = L_{\bf} y(\bx) = \omega}$,
${L_{\bg} y(\bx) = 0}$,
and ${L_{\bg}L_{\bf}y(\bx) = 1}$, $y$ has relative degree two and satisfies Assumption~\ref{assum:reldeg}. 
Furthermore, $\psi$ satisfies Assumption~\ref{assum:singleint} for ${\E=\R^2}$, because ${\derp{\psi}{y}(\varphi)=-2\varphi=0}$ holds for ${\varphi=0}$ and ${\psi(0)=\pi^2/4>0}$.

\subsubsection{High-order CBF}

The HOCBF in~\eqref{eq:hocbf} becomes:
\begin{equation}
    h(\bx) = -2 \varphi \omega + \alpha \bigg( \frac{\pi^2}{4} - \varphi^2 \bigg),
    \label{eq:pend_hocbf}
\end{equation}
with ${L_{\bf} h(\bx) = -2 \omega^2 - 2 \varphi \sin(\varphi) - 2 \alpha'(\pi^2/4-\varphi^2) \varphi \omega}$ and
${L_{\bg} h(\bx) = - 2 \varphi}$. If ${L_{\bg} h(\bx) = 0}$, we have ${\varphi = 0}$, leading to ${L_{\bf} h(\bx) = -2 \omega^2}$, and ${h(\bx) = \alpha(\pi^2/4)}$.
Thus,~\eqref{eq:Lgh=0} does not hold for states where ${\omega^2 \geq \alpha(\alpha(\pi^2/4))/2}$.
Therefore, based on Lemma~\ref{lemma:cbf}, this choice of $h$ is not a CBF.


\subsubsection{Rectified CBF}

The ReCBF in~\eqref{eq:ReCBF} takes the form:
\begin{equation}
    h(\bx) \!=\! \frac{\pi^2}{4} - \varphi^2 - \frac{1}{2 \mu} {\rm ReQU} \bigg( \!2 \varphi \omega - \alpha \bigg(\! \frac{\pi^2}{4} - \varphi^2 \!\bigg) + \varepsilon \!\bigg)\!,
    \label{eq:pend_recbf}
\end{equation}
with the choice ${\gamma(s) = -s^2/(2\mu)}$ for ${s < 0}$.
Based on~\eqref{eq:ReCBF_condition}, this is a CBF if ${\alpha(\pi^2/4) \geq \varepsilon}$, putting an upper bound on $\varepsilon$.
\subsubsection{Backstepping CBF}
The first step of this method is to design the virtual controller $\kappa$ so as to ensure~\eqref{eq:topdyn}:
\begin{equation}
    - 2 \varphi \kappa(\varphi) > -\alpha \bigg( \frac{\pi^2}{4} - \varphi^2 \bigg).
\end{equation}
For example, we can choose ${\kappa(\varphi) = -K \varphi}$ and ${\alpha(r) = \alpha_{\rm c} r}$ with ${2 K \geq \alpha_{\rm c} > 0}$.
With this, the backstepping CBF~\eqref{eq:backstepping} is:
\begin{equation}
    h(\bx) = \frac{\pi^2}{4} - \varphi^2 - \frac{1}{2\mu} \big( \omega + K \varphi \big)^2.
\label{eq:pend_backstepping}
\end{equation}
The corresponding safe set $\S$ is a rotated ellipse in the state space.
Elliptical safe sets have previously been constructed by intuition for the inverted pendulum~\cite{Alan2023AV} and also for lane keeping control~\cite{jiang2024safety}---these can also be viewed as a special case of backstepping with a linear virtual controller $\kappa$. 

\subsubsection{Activated Backstepping CBF}
We use the
ABC in~\eqref{eq:newCBF} with ${\gamma(s) = s^2/(2\mu)}$ for ${s > 0}$:
\begin{equation}
\begin{aligned}
    h(\bx) & = 
    \begin{cases}
        \frac{\pi^2}{4}-\varphi^2
        & \text{if } s(\bx) \geq 0, \\
        \frac{\pi^2}{4}-\varphi^2 - \frac{s(\bx)^2}{2\mu}
        & \text{if } s(\bx) < 0, \\
    \end{cases} \\
    s(\bx) & = -2\varphi (\omega+K\varphi),
\label{eq:pend_activated_switching}
\end{aligned}
\end{equation}
or, more compactly, as in~\eqref{eq:activated}, with ${\Theta(s) = {\rm ReQU}(s)/(2\mu)}$:
\begin{equation}
    h(\bx) = \frac{\pi^2}{4} - \varphi^2 - \frac{1}{2 \mu} {\rm ReQU} \big( 2\varphi (\omega+K\varphi) \big).
\label{eq:pend_activated}
\end{equation}
Notice the difference between the standard backstepping in~\eqref{eq:pend_backstepping} and the activated version in~\eqref{eq:pend_activated_switching}: we incorporated the switching function $s$ that includes a factor $\derp{\psi}{y}(y(\bx)) = -2 \varphi$.

Figure~\ref{fig:pendulum} shows numerical simulation results for the inverted pendulum~\eqref{eq:pendulum} with the safety-critical controller~\eqref{eq:safetyfilter}-\eqref{eq:lambda_QP} using: (a,f) HOCBF~\eqref{eq:pend_hocbf}, (b,g) ReCBF~\eqref{eq:pend_recbf}, (c,h) backstepping CBF~\eqref{eq:pend_backstepping}, and
(d,i) ABC~\eqref{eq:pend_activated}.
The pendulum starts moving from the same initial condition in each case, the desired input is ${k_{\rm d}(\bx) = 0}$, and the CBF must keep the angle within ${-\pi/2 \leq \varphi \leq \pi/2}$.
The extended class-$\K$ function of each CBF is linear, ${\alpha(r) = \alpha_{\rm c} r}$, and the parameters are:
${\alpha_{\rm c} = 1\,{\rm s^{-1}}}$ and ${\Gamma=1}$ for all CBFs,
${\varepsilon=2\,{\rm s^{-1}}}$ for ReCBF,
${K = 0.75\,{\rm s^{-1}}}$ for both backstepping methods,
${\mu = 1.5\,{\rm s^{-2}}}$ for backstepping CBF, and 
${\mu = 5\,{\rm s^{-2}}}$ for ABC.

Figure~\ref{fig:pendulum}(a,b,c,d) show the phase plane with the constraint set $\C$ (red), safe set $\S$ (green), and simulated trajectory (thick blue).
All four cases are safe: the trajectories stay inside ${\S \cap \C}$ for the HOCBF and ${\S \subset \C}$ for ReCBF, backstepping CBF, and ABC.
Correspondingly, the CBF $h$ and constraint function $\psi$ in Fig.~\ref{fig:pendulum}(f,g,h,i) are positive for all time.

Figure~\ref{fig:pendulum}(e,j) compare the control inputs of the four methods.
The backstepping input (purple) has a larger magnitude than that of the HOCBF (blue), while the activated backstepping input (orange) is slightly larger and changes more sharply.
The input of the ReCBF with ${\varepsilon=2\,{\rm s^{-1}}}$ (green) is similar to that of the ABC.
However, selecting $\varepsilon$ too large, such as ${\varepsilon=4\,{\rm s^{-1}}}$, violates~\eqref{eq:ReCBF_condition} (i.e., ${\alpha(\pi^2/4) < \varepsilon}$), thus $h$ is not a valid CBF, and the input blows up (red dashed).
If one chooses a too small ${\varepsilon \approx 0}$ value, such as ${\varepsilon=0.01\,{\rm s^{-1}}}$, then the input has a sharp change (jump) at ${t \approx 2\,{\rm s}}$ and chatters at ${t \in [3,4]\,{\rm s}}$ (red), which become more pronounced for even smaller $\varepsilon$, underscoring the challenges in tuning $\varepsilon$ for ReCBFs;
see~Remark~\ref{rem:comparison}.

Figure~\ref{fig:pendulum}(a,b,c,d) also show the singular region ${L_{\bg}h(\bx) \!=\! 0}$ (thin black) and switching line ${L_{\bf}h(\bx) + \alpha(h(\bx)) = 0}$ where the controller starts to deviate from the desired ${k_{\rm d}(\bx) = 0}$ (dash-dot magenta); cf.~\eqref{eq:safetyfilter}-\eqref{eq:lambda_QP}.
A CBF is valid if the singular region ${L_{\bg}h(\bx) = 0}$ lies between the switching lines where ${L_{\bf}h(\bx) > - \alpha(h(\bx))}$; see~\eqref{eq:Lgh=0}.
This is true for ReCBF, backstepping CBF, and ABC, but not for the HOCBF that violates~\eqref{eq:Lgh=0} at high angular velocities (i.e., at ${\omega^2 \geq \alpha(\alpha(\pi^2/4))/2}$ shown by the dotted black line).
Along trajectories with high angular velocity (like the thick dashed blue line in Fig.~\ref{fig:pendulum}(a)), the control input blows up (see the thick dashed blue line in Fig.~\ref{fig:pendulum}(e)), which prevents the implementation of the HOCBF.

For backstepping in Fig.~\ref{fig:pendulum}(c), the safe set $\S$ is compact (an ellipse).
For activated backstepping in Fig.~\ref{fig:pendulum}(d), $\S$ is unbounded, and the boundaries of $\S$ and $\C$ coincide in the top left and bottom right quadrants (black) where the angular velocity is safe (i.e., where the pendulum moves towards the upright position).
Thus, ${h(\bx) = \psi(y(\bx))}$ in Fig.~\ref{fig:pendulum}(i) at the beginning of the simulation when the trajectory is in the top left quadrant in Fig.~\ref{fig:pendulum}(d).
This corresponds to ${s(\bx) \geq 0}$ according to~\eqref{eq:pend_activated_switching}, 
and also to ${L_{\bg}h(\bx)=0}$ based on~\eqref{eq:activated_Lgh_vs_s}.

Overall, activated backstepping yields a valid CBF 
under weaker conditions than
the HOCBF method.
The ABC aims to minimize the difference between the safe set and the constraint set, yielding a larger safe set than traditional backstepping, at the price of a slightly larger and sharper input.
Furthermore, the ABC generates a Lipschitz continuous controller without jumps or chatter, a phenomenon that may occur for ReCBFs if $\varepsilon$ is not chosen carefully.

\section{APPLICATION TO SAFE VEHICLE CONTROL}
\label{sec:application}

Finally, we implement activated backstepping for the safe navigation of automated vehicles.
We design safe controllers for obstacle avoidance using a kinematic bicycle model \cite{haraldsen2024safety}:
\begin{align}
    \underbrace{\begin{bmatrix}
        \dot{\xi} \\
        \dot{\eta} \\
        \dot{\vartheta} \\
        \dot{v}
    \end{bmatrix}}_{\dot{\bx}}
    = \underbrace{\begin{bmatrix}
        v\cos(\vartheta) \\
        v\sin(\vartheta) \\
        0 \\
        0
    \end{bmatrix}}_{\bf(\bx)}
    + \underbrace{\begin{bmatrix}
        0 & 0 \\
        0 & 0 \\
        v/L & 0 \\
        0 & 1
    \end{bmatrix}}_{\bg(\bx)}
    \underbrace{\begin{bmatrix}
        u_1 \\
        u_2
    \end{bmatrix}}_{\bu},
    \label{eq:bicycle}
\end{align}
where $L$ is the wheelbase.
The vehicle has four states: rear wheel position ($\xi$, $\eta$), yaw angle $\vartheta$, and rear wheel speed $v$; while it has two control inputs: tangent of steering angle $u_1$ and longitudinal acceleration $u_2$.
We assume that ${v \neq 0}$.

As shown in Fig.~\ref{fig:bicycle}(a), our goal is to drive the vehicle with desired velocity ${v_{\rm d}>0}$ along a lane at ${\eta=0}$, and then safely bypass a circular obstacle of radius $R_{\rm O}$ centered at ($\xi_{\rm O}$, $\eta_{\rm O}$).
Collision-free behavior is captured by the output coordinates $\by(\bx)=(\xi,\eta)$ and constraint function: 
\begin{equation}
    \psi(\by(\bx)) = (\xi - \xi_{\rm O})^2 + (\eta - \eta_{\rm O})^2 - R_{\rm O}^2,
\end{equation}
which satisfy Assumption~\ref{assum:reldeg} if ${v \neq 0}$ and Assumption~\ref{assum:singleint} for ${\E=\{ \bx \in \X: (\xi,\eta) \neq (\xi_{\rm O},\eta_{\rm O})\}}$.

We design a safe controller based on the
ABC~\eqref{eq:activated}.
First, we construct a virtual controller $\bkappa$ satisfying~\eqref{eq:topdyn}.
While this is more challenging than in the previous example, it can be done algorithmically using the smooth safety filter in~\eqref{eq:safetyfilter} and~\eqref{eq:lambda_smooth} for the single integrator (with weights ${\hat{\Gamma}_1 \!=\! \hat{\Gamma}_2 \!=\! 1}$):
\begin{equation}
\begin{aligned}
    & \bkappa(\by) = \bkappa_{\rm d}(\by) + \lambda \big( a(\by), \|\bb(\by)\|^2 \big) \bb(\by),\\
    & a(\by) = \dot{\psi} \big( \by, \bkappa_{\rm d}(\by) \big) + \hat{\alpha}_{\rm c} \psi(\by), \quad
    \bb(\by) = \derp{\psi}{\by}(\by),
    \label{eq:obst_av_k0}
\end{aligned}
\end{equation}
where ${\bkappa_{\rm d}(\by)=(\hat{v}_{\rm d}, 0)}$ is chosen to be a constant velocity with ${\hat{v}_{\rm d}>0}$.
Then, we use the proposed CBF~\eqref{eq:activated} with ${\Theta(s) = {\rm ReQU}(s)/(2\mu)}$, and synthesize a safe controller via the safety filter~\eqref{eq:safetyfilter}-\eqref{eq:lambda_QP} with ${\alpha(r) = \alpha_{\rm c} r}$.
We set the desired controller to be a lane-keeping controller inspired by \cite{jiang2024safety}:
\begin{equation}
    \bk_{\rm d}(\bx)=\begin{bmatrix}
        -K_{\eta}\eta - K_{\vartheta}\sin(\vartheta)\\
        K_v(v_{\rm d}-v)
    \end{bmatrix}.
\end{equation}
with gains ${K_{\eta}, K_{\vartheta}, K_v > 0}$.

Figure~\ref{fig:bicycle} shows simulations of the vehicle~\eqref{eq:bicycle} executing controller~\eqref{eq:safetyfilter}-\eqref{eq:lambda_QP} with the proposed CBF~\eqref{eq:activated} and the parameters in Table~\ref{tab:parameters}.
The trajectory in Fig.~\ref{fig:bicycle}(a) shows that the vehicle successfully bypasses the obstacle without collision, while the positive signs of $\psi$ and $h$ in Fig.~\ref{fig:bicycle}(d) imply safety.
Initially, the desired controller is used (see Fig.~\ref{fig:bicycle}(b,e)), the yaw angle is zero while the speed increases towards the desired value $v_{\rm d}$ (Fig.~\ref{fig:bicycle}(c,f)), and the CBF $h$ matches the constraint function $\psi$ (Fig.~\ref{fig:bicycle}(d)).
As the vehicle approaches the obstacle, $h$ and $\psi$ deviate, indicating that the situation is safety-critical.
Then, the vehicle responds to the obstacle and the input differs from the desired value (Fig.~\ref{fig:bicycle}(b,e)).
After passing the obstacle, the desired input is used again.
Driving away from the obstacle, $h$ matches $\psi$ again because the vehicle's velocity is safe.
Finally, the yaw angle converges to zero and the speed approaches $v_{\rm d}$ (Fig.~\ref{fig:bicycle}(c,f)) as the vehicle continues to travel in its lane as desired.






\begin{table}
\caption{Parameters of the vehicle control example}
\vspace{-5mm}
\begin{center}
\begin{tabular}{|c | c | c| |c | c | c| |c | c | c|} 
\hline
Par.\! & \!Value\! & \!Unit &
Par.\! & \!Value\! & \!Unit &
Par.\! & \!Value\! & \!Unit \\
\hline
$L$ & $2.5$ & m &
$v_{\rm d}$ & $10$ & m/s &
$\hat{v}_{{\rm d}}$ & $4$ & m/s \\ 
$\xi_{\rm O}$ & $20$ & m &
$K_{\eta}$ & $0.4$ & \!1/(ms)\! &
$\hat{\alpha}_{\rm c}$ & $1$ & 1/s \\
$\eta_{\rm O}$ & $-0.1$ & m &
$K_{\vartheta}$ & $1.75$ & 1/s &
$\sigma$ & \!$0.001$\! & 1/s$^2$ \\ 
$R_{\rm O}$ & $4$ & m &
$K_v$ & $0.3$ & 1/s &
$\mu$ & $1$ & \!m$^2$/s$^2$\! \\ 
$\Gamma_1$ & $1$ & 1 &
$\Gamma_2$ & $0.15$ & \!s$^4$/m$^2$\! &
$\alpha_{\rm c}$ & $5$ & 1/s \\
\hline
\end{tabular}
\end{center}
\vspace{-5mm}
\label{tab:parameters}
\end{table}

\section{CONCLUSION}
\label{sec:concl}

We proposed {\em activated backstepping} to constructively generate control barrier functions (CBFs) for safety-critical control.
We showed that activated backstepping CBFs (ABCs) are valid CBFs that provide advantages over existing high-order, rectified, or backstepping CBFs in certain scenarios.
We highlighted these advantages in a comparative example of an inverted pendulum.
Finally, we demonstrated the efficacy of the proposed ABC in safety-critical vehicle control for collision-free driving.


% ...

% ...


\bibliographystyle{IEEEtran}
\bibliography{lcss_cdc_2025}	

\end{document}